\documentclass[11pt]{amsart}
\usepackage{xcolor}
\newtheorem{theorem}{Theorem}[section]

\newtheorem{lemma}[theorem]{Lemma}

\theoremstyle{definition}

\theoremstyle{remark}

\numberwithin{equation}{section}

\begin{document}

\title[Solutions for autonomous semilinear elliptic equations]{Solutions for autonomous semilinear elliptic equations}

\author[Alexis Molino]{Alexis Molino}
\address{Alexis Molino\textsuperscript{1} ---
Departamento de Matem\'aticas, Facultad de Ciencias Experimentales, Universidad de Almer\'{\i}a, Ctra. de Sacramento sn. 04120, La Ca\~{n}ada de San Urbano, Almer\'{\i}a, Spain.}
\email{amolino@ual.es}

\author[Salvador Villegas]{Salvador Villegas}
\address{Salvador Villegas\textsuperscript{2} ---
 Departamento de An\'alisis Matem\'atico, Universidad de Granada, 18071 Granada, Spain. }
\email{svillega@ugr.es}

\begin{abstract}
We study existence of nontrivial solutions to problem
\begin{equation*}
\left\lbrace
\begin{array}{rcll}
-\Delta u &=& \lambda u+f(u)&\text{ in }\Omega,\\
u&=&0&\text{ on }\partial \Omega,
\end{array}\right.
\end{equation*}
where $\Omega \subset \mathbb{R}^N$ is a smooth bounded domain, $N\geq 1$, $\lambda \in \mathbb{R}$ and  $f:\mathbb{R}\to \mathbb{R}$ is any locally Lipschitz function with nonpositive primitive. \\ A complete description is obtained for $N=1$ and partial results for $N\geq 2$.
\end{abstract}

\renewcommand{\thefootnote}{\fnsymbol{footnote}} 
\renewcommand{\thefootnote}{\arabic{footnote}} 

\maketitle
\section{Introduction}

In this paper we deal about the existence of nontrivial solutions for the following autonomous semilinear elliptic differential equation

\begin{equation}\label{problema}
\left\lbrace
\begin{array}{rcll}
-\Delta u &=& \lambda u+f(u)&\text{ in }\Omega,\\
u&=&0&\text{ on }\partial \Omega,
\end{array}\right.
 \tag{$P$}
\end{equation}
being $\Omega \subset \mathbb{R}^N$  a smooth bounded domain, $N\geq 1$, $\lambda \in \mathbb{R}$ and $f:\mathbb{R}\to \mathbb{R}$ a locally Lipschitz function satisfying

\begin{equation}\label{hipotesis}
F(s):=\int_0^s f(t)dt\leq 0 \mbox{ for every }s\in \mathbb{R}.
 \tag{$H$}
\end{equation}
Typical example are $f(u)=-\sin u$ or $f(u)=-u(u-1)(u-2)$

  A bounded weak solution (solution from now on) to problem \eqref{problema} is a function $u\in H_0^1(\Omega)\cap L^\infty(\Omega)$ such that
  \begin{equation*}
  \int_\Omega \nabla u \nabla \varphi = \lambda \int_\Omega u\varphi +\int_\Omega f(u)\varphi,\quad \textit{for all} \,\, \varphi \in H_0^1(\Omega).
  \end{equation*}
  
 Observe that, by regularity results, every solution is a classical solution to this problem.

 Note that hypothesis \eqref{hipotesis} implies that $u\equiv 0$ is always a solution to problem \eqref{problema}.  
   In \cite{lamadre} the authors establish the uniqueness of trivial solution for the case $\lambda=0$. Therefore, we are interested in existence of nontrivial solutions by adding the  term $\lambda u$ to the source data $f(u)$.
 As will be seen below, this will depend both on the dimension of the space and on the different values of
 $\lambda$. Besides, $\lambda_1$ (the first eigenvalue of the Laplace operator with Dirichlet boundary conditions on $\partial\Omega$) plays a crucial role. That is, an effect similar to that found in the well-known Br\'ezis-Nirenberg work \cite{BN}.
 
   However, as can be seen in the next section, partial results will be obtained, leaving open questions reflected in the last section.

 \section{Main results}
 
\begin{lemma}\label{general} Let $N\geq 1$. If $\lambda \leq 0$, then $u\equiv 0$ is the unique solution to \eqref{problema}, while for $\lambda\geq \lambda_1$ there exists $f$ satisfying \eqref{hipotesis} such that there exists a nontrivial solution to \eqref{problema}.
\end{lemma}

\begin{proof}
Case $\lambda \leq 0$ is a fairly straightforward result of \cite{lamadre}, since the term $\lambda u+f(u)$ has nonpositive primitive. On the other hand, for case $\lambda\geq \lambda_1$ it is sufficient to take $f(u)=(\lambda_1-\lambda)u$ and observe that $u=\varphi_1$ (the first eigenfunction of the Laplace operator with Dirichlet boundary conditions on $\partial\Omega$) is a classical solution to \eqref{problema}.
\end{proof}
The previous lemma focuses the problem for values of  $\lambda\in (0,\lambda_1)$. Next, a partial result is obtained for dimensions greater than $2$.

\begin{theorem}\label{N>2}

Let $N\geq 3$ and $\Omega\subset\mathbb{R}^N$  a smooth bounded starshaped domain and $\lambda \leq \frac{N-2}{N}\lambda_1$.  Then, problem (\ref{problema}) has no nontrivial solutions.
\end{theorem}
\begin{proof}
Without loss of generality we can assume, up to a translation, that   $\Omega$ is starshaped respect to the origin. Therefore  $x\cdot \nu(x)\geq 0$ on $\partial \Omega$, where $\nu$ denotes the unit  vector normal to $\partial \Omega$ pointing outwards.

By the well-known Poho\v{z}aev identity \cite{Pohozaev}, every solution $u$ to \eqref{problema} must satisfy
 \begin{align*}\label{Poho}
\frac{1}{2}\int_{\partial \Omega}|\nabla u(x)|^2\,x\cdot \nu(x)\,dx+\frac{N-2}{2}\int_{\Omega}|\nabla u(x)|^2\,dx  
\\
=\lambda\,\frac{ N}{2}\int_\Omega u(x)^2\,dx+N\int_\Omega F(u(x))\,dx,
\end{align*}

which implies
  \begin{equation*}
  \frac{N-2}{2}\int_{\Omega}|\nabla u(x)|^2\,dx\leq\lambda\,\frac{ N}{2}\int_\Omega u(x)^2\,dx+N\int_\Omega F(u(x))\,dx,
  \end{equation*}  
  and by using Poincar\'e inequality
  
   \begin{equation*}
 \lambda_1\, \frac{N-2}{2}\int_{\Omega}u(x)^2\,dx\leq\lambda\,\frac{ N}{2}\int_\Omega u(x)^2+N\int_\Omega F(u(x))\,dx.
  \end{equation*}

  We claim that $\int_\Omega F(u(x))\,dx<0$.  Conversely, since $f$ satisfies \eqref{hipotesis} and $F\circ u$ is continuous, then $F(u(x))=0$ for every $x\in \Omega$. Therefore $F(s)=0$ for every $s\in I=\{ u(x): x\in\Omega\}$. Since $u$ is nontrivial it follows that $I$ is a nontrivial interval and hence $f(s)=F'(s)=0$  for every $s\in I$. Thus $u$ satisfies $-\Delta u=\lambda u$ in $\Omega$ and consequently $\lambda$ is an eigenvalue of the Laplace operator with Dirichlet boundary conditions on $\partial\Omega$, contradicting hypothesis $\lambda \leq \frac{N-2}{N}\lambda_1$.
  
  Finally, from the above inequality it follows that

  \begin{equation*}
 \lambda_1\, \frac{N-2}{2}\int_{\Omega}u(x)^2\,dx<\lambda\,\frac{ N}{2}\int_\Omega u(x)^2\,dx.
  \end{equation*}
This leads to the fact that any nontrivial solution to the problem \eqref{problema} must satisfy
$$
\lambda >  \lambda_1\, \frac{N-2}{N}
$$
and the proof is concluded. 
\end{proof}

 \
 However, as can be seen in the following theorem, in dimension $1$ a drastic change occurs: the existence of nontrivial solutions for every positive $\lambda$.
 \begin{theorem}\label{N=1}

Let $\Omega=(a,b)\subset \mathbb{R}$ and set $\lambda>0$. Then, there exists $f\in C^\infty (\mathbb{R})$ satisfying (\ref{hipotesis}), such that problem (\ref{problema}) admits a nontrivial solution $u\in C^\infty [a,b]$.

\end{theorem}
For its proof, we will first state the following technical lemma:

\begin{lemma}\label{laidea}

Let $a<b$ and $\lambda$ real numbers satisfying $0<\lambda<\pi^2/(b-a)^2$. Then, there exists a unique $M>0$, depending on $\lambda$, such that 

$$\int_0^1 \frac{ds}{\sqrt{Ms^2(1-s)^2+\lambda (1-s^2)}}=\frac{b-a}{2}.$$

\end{lemma}

\begin{proof}

Define the function $\Phi:(0,+\infty)\to \mathbb{R}$ by

$$\Phi(k):=\int_0^1 \frac{ds}{\sqrt{k s^2(1-s)^2+\lambda (1-s^2)}}, \ \ \ k>0.$$

It is obvious that

$$\frac{1}{\sqrt{k s^2(1-s)^2+\lambda (1-s^2)}}\leq \frac{1}{\sqrt{\lambda (1-s^2)}}\in L^1(0,1) \mbox{  for every }k>0.$$

Hence, applying the Convergence Dominated Theorem we can assert that $\Phi$ is a well defined decreasing continuous function satisfying

$$\lim_{k\to +\infty}\Phi(k)=\int_0^1 0\, ds=0,$$

$$\lim_{k\to 0}\Phi(k)=\int_0^1 \frac{ds}{\sqrt{\lambda (1-s^2)}}=\frac{\pi}{2\sqrt{\lambda}}.$$

Applying  $0<\lambda<\pi^2/(b-a)^2$, we deduce that 

$$\frac{b-a}{2}\in \left(\lim_{k\to +\infty}\Phi(k), \lim_{k\to 0}\Phi(k)\right),$$

\noindent which  gives $\Phi(M)=(b-a)/2$ for a unique $M>0$, and the proof is complete.

\end{proof}

{\it Proof of Theorem \ref{N=1}}

\

If $\lambda\geq \lambda_1=\pi^2/(b-a)^2$, by Lemma \ref{general} there exists a nontrivial solution to problem \eqref{problema} for $f(s)=\left(\pi^2/(b-a)^2 -\lambda \right) s,  \,\,s\in \mathbb{R}$.

For the rest of the proof we will suppose that  $0<\lambda<\pi^2/(b-a)^2$.

Define the function $f:\mathbb{R}\to \mathbb{R}$ by

$$f(s)=-Ms(s-1)(2s-1), \ \ \ s\in \mathbb{R},$$

\noindent where $M>0$ is the value obtained in Lemma \ref{laidea}. It is immediate that $f\in C^\infty(\mathbb{R})$ and satisfies $F(s):=\int_0^s f(t)dt=-Ms^2(1-s)^2/2\leq 0$, for every $s\in\mathbb{R}$. 

  Now, we will construct a nontrivial solution $u\in C^\infty (a,b)$ to problem (\ref{problema}).

For this purpose; define $\Psi:[0,1]\to [a,(a+b)/2]$ by
$$\Psi(t):=a+\int_0^t \frac{ds}{\sqrt{Ms^2(1-s)^2+\lambda (1-s^2)}}, \ \ \ t\in [0,1].$$

Clearly $\Psi$ is a $C^\infty [0,1)\cap C[0,1]$ bijection between $[0,1]$ and $[a,(a+b)/2]$, with $\Psi'(t)=1/\sqrt{Mt^2(1-t)^2+\lambda (1-t^2)}>0$, $t\in [0,1).$

Therefore, we can define $u:=\Psi^{-1}$, obtaining $u\in C^\infty [a,(a+b)/2)\cap C[a,(a+b)/2]$ with
$$u'(x)=\sqrt{Mu(x)^2(1-u(x))^2+\lambda (1-u(x)^2)}, \ \ \  x\in [a,(a+b)/2).$$

Since $u\left( (a+b)/2\right)=1$, we obtain that $\lim_{x\to (a+b)/2}u'(x)=0$ and consequently $u\in C^1[a,(a+b)/2]$.

   Moreover, from the above expression we can assert that
$$u'(x)^2=Mu(x)^2(1-u(x))^2+\lambda (1-u(x)^2) , \ \ \  x\in [a,(a+b)/2),$$

\noindent and differentiating yields
$$2u'(x)u''(x)=u'(x)\left( 2M u(x)(u(x)-1)(2u(x)-1)-2\lambda u(x)\right), \ \ \ x\in [a,(a+b)/2).$$

Since $u'(x)\neq 0$ for $x\in  [a,(a+b)/2)$ we conclude that 
$$-u''=f(u)+\lambda u,\,\, x\in [a,(a+b)/2).$$

   In fact, the continuity of $f$ implies that there exists $\lim_{x\to (a+b)/2}u''(x)$, and consequently $u\in C^2[a,(a+b)/2]$.

Finally, we define $u$ for  $x\in((a+b)/2, b]$ by symmetry. That is, 
$$u(x):=u(a+b-x),\,\, x\in((a+b)/2, b].$$ 
Taking into account that $u(a)=0$, $u'((a+b)/2)=0$ and $u\in C^2[a,(a+b)/2]$. It follows that $-u''=f(u)+\lambda u$ in $[a,(a+b)/2]$. So,  we deduce that $u\in C^2[a,b]$ is a solution of problem (\ref{problema}). 

  Finally, since $f\in C^\infty (\mathbb{R})$ we conclude $u\in C^\infty [a,b]$ and the result follows.

\qed

\section{Summary and open problems}
Our results are summarized as follows. Uniqueness/non-uniqueness of trivial solution to problem \eqref{problema}. 

\begin{itemize}
\item $N=1$; Uniqueness for $\lambda\leq 0$ and non-uniqueness for $\lambda> 0$.
\item $N=2$; Uniqueness for $\lambda\leq 0$ and non-uniqueness for $\lambda \geq \lambda_1$.
\item $N\geq3$; Uniqueness for $\lambda\leq \lambda_1\frac{N-2}{N}$ (in case that the domain is starshaped) and non-uniqueness for $\lambda\geq \lambda_1$.
\end{itemize}

We finish our discussion with the following open questions:

\begin{enumerate}
\item Does \eqref{problema} admit a nontrivial solution if $N=2$ and $0<\lambda<\lambda_1$?
\item Does \eqref{problema} admit a nontrivial solution if $N\geq 3$ and $\lambda_1\frac{N-2}{N}<\lambda<\lambda_1$?
\end{enumerate}

\section*{Acknowledgements}
The authors are supported  by  supported by (MCIU) Ministerio de Ciencia, Innovaci\'on y Universidades, Agencia Estatal de
Investigaci\'on (AEI) and Fondo Europeo de Desarrollo Regional under Research Project
PID2021-122122NB-I00 (FEDER) and by Junta de Andaluc\'ia FQM-116 (Spain).

\

%
%
%
%


\begin{thebibliography}{99}

\bibitem{BN}{Br\'ezis, H.; Nirenberg, L.,}{ \it Positive solutions of nonlinear elliptic equations involving critical Sobolev exponents,}{ Commun. Pure Appl. Math. 36 (1983), 437-477.}

\bibitem{lamadre} {L\'opez-Mart\'{\i}nez, S.; Molino, A.,} {\it Nonexistence result of nontrivial solutions to the equation $-\Delta u=f(u)$,} {Complex Var. Elliptic Equ. 67 (2022), no. 1, 239-245.}
\bibitem{Pohozaev} {Poho\v{z}aev, S. I.,} {\it On the eigenfuctions of the equation $\Delta u+\lambda f(u)=0$,} {Dokl. Akad. Nauk SSSR, 165 (1965), 36-39.}






\end{thebibliography}
\end{document}